\documentclass[12pt]{amsart}
\usepackage{amsfonts}
\usepackage{amsmath}
\usepackage{amssymb}
\usepackage{graphicx}
\usepackage{amscd}
\usepackage{xcolor}
\usepackage{hyperref}

\DeclareMathOperator\arctanh{arctanh}
\newtheorem{theorem}{Theorem}
\newtheorem{acknowledgement}[theorem]{Acknowledgement}

\newtheorem{definition}[theorem]{Definition}
\newtheorem{example}[theorem]{Example}

\newtheorem{lemma}[theorem]{Lemma}

\newtheorem{proposition}[theorem]{Proposition}

\newcommand{\sn}{\mathrm{sn}}
\newcommand{\cn}{\mathrm{cn}}
\newcommand{\dn}{\mathrm{dn}}

\begin{document}

	\author{G. Polychrou, E. Papageorgiou, A. Fotiadis, C. Daskaloyannis }
	
	\title[New examples of harmonic maps to the hyperbolic plane]{New examples of harmonic maps to the hyperbolic plane via B{\"a}cklund transformation}
	
	\maketitle

	%the classification is valid for 2010

	\begin{abstract}
		We study harmonic maps from a subset of the complex plane to a subset of the hyperbolic plane. In \cite{FotDask}, harmonic maps are related to the sinh-Gordon equation and a B{\"a}cklund transformation is introduced, which connects solutions of the sinh-Gordon and sine-Gordon equation. We develop  this machinery in order to construct new harmonic maps to the hyperbolic plane. 
	\end{abstract}
	
	\maketitle
	
	\section{Introduction and Statement of the Results}\label{Introduction}

	This article has been motivated by the following open problem: how can we construct a harmonic map explicitly?
	
	The construction of harmonic maps from a Riemann surface to the hyperbolic plane is a classical problem. Many beautiful results have been achieved in this direction, for example, Wolf's parametrization of Teichm{\"u}ller space via harmonic maps \cite{Wolf1}, the study of Au-Tam-Wan \cite{ATW} and others of harmonic maps to ideal polygons with polynomial Hopf differential, the construction of a harmonic diffeomorphism from the complex plane which solved a conjecture of Schoen \cite{CR}, and so on.
	
	In this paper, our aim is to construct harmonic maps between Riemann surfaces when the curvature in the target is a negative constant, say $-1$. We should emphasize that our construction is local, thus we consider an open simply connected set that contains the origin and we consider the target to be a subset of the hyperbolic plane $\mathbb{H}^2$. For that, we develop some of the machinery introduced in \cite{FotDask}, which connects the harmonic map problem with elliptic versions of the sinh-Gordon and sine-Gordon equation.
	
	Let $w=w(x,y)$ be a solution of the sinh-Gordon equation \begin{equation}\label{sinh-Gordon eqn}
		\Delta w=2 \sinh(2w)
	\end{equation} 
	where $\Delta=\partial_{xx}^2+\partial_{yy}^2$ is the Laplacian with the flat metric and let  $u=u(z, \bar{z})$ be a solution of the Beltrami equation 
	\begin{equation}\label{Beltrami eqn}
		\frac{\partial_{\bar{z}}u}{\partial_{z}u}=e^{-2w},
	\end{equation} 
	and $z=x+iy$ lie in an open simply connected subset $\Omega$ of $\mathbb{C}$ where the map $u$ is a well defined $C^{2}$ map. Without loss of generality we assume that $\Omega$ contains the origin. Then, $u$ is a harmonic map, if the curvature of the target is $-1$ \cite[Theorem 1, Corollary 2]{FotDask}.
		\begin{definition}
			For $w$ and $u$ as above, we say that $u$ is the harmonic map that corresponds to the solution $w$ of the sinh-Gordon (\ref{sinh-Gordon eqn}).
	\end{definition}
	From now on, consider the target  surface fixed and realized as the hyperbolic half plane $\mathbb{H}^2$. Using a specific coordinate system  on the domain (see Section \ref{sec:Preliminaries}  for more details),  the harmonic map equation under study becomes
	\begin{equation}\label{eq:HarmHyper}
		\frac{\partial_z u \partial_z\bar{u}}{S^2}=1.
	\end{equation}
This is possible because the Hopf differential is assumed to be non vanishing.

	In \cite{FotDask}, it is proved that the system
	\begin{align}
		\partial_x w-\partial_y \theta&=-2\sinh w \sin \theta \label{Back1intro}\\
		\partial_y w+\partial_x\theta&=-2\cosh w \cos \theta\label{Back2intro},
	\end{align}	
	is a \textit{B{\"a}cklund transformation} that connects a solution $w=w(x,y)$ of the sinh-Gordon equation (\ref{sinh-Gordon eqn}) and a solution $\theta=\theta(x,y)$ of the sine-Gordon equation 
	\begin{equation}	\label{sin-Gordon eqn}
		\Delta \theta=-2 \sin(2\theta).
	\end{equation}
This motivates us to define the following class of functions.
	\begin{definition}
			We say that the pair of functions $(w,\theta)$ is in the class $(BT)$,  if the functions $w$ and $\theta$ satisfy (\ref{Back1intro})-(\ref{Back2intro}). Consequently, $w$ is a solution of the sinh-Gordon equation (\ref{sinh-Gordon eqn}) and $\theta$ is a solution of the sine-Gordon equation (\ref{sin-Gordon eqn}).
		\end{definition}

	From now on we assume that $(w,\theta)\in (BT)$. 
	Set 
	\begin{align*}I_1&=I_1(x)=\int_0^x\cosh w(t,0)\sin\theta(t,0)dt, \\
		I_2&=I_2(x,y)=\int_0^y\sinh w(x,s)\cos\theta(x,s)ds\\
		I_3&=I_3(x)=\int_0^x e^{2 I_1 (t)}\cosh w(t,0)\cos\theta(t,0)dt,\\
		I_4&=I_4(x,y)=e^{2 I_1 (x)}\int_0^y e^{2 I_2 (x,s)}\sinh w(x,s)\sin\theta(x,s)ds.
	\end{align*}
	We first prove the following result.
	\begin{proposition}\label{Theorem2} Define the function $S$ by 
		\begin{align*}
			S(x,y)=S(0,0)e^{2(I_1+I_2)}, 
		\end{align*} 
		and the function $R$ by
		\begin{align*}
			R(x,y)=R(0,0)+2S(0,0)(I_3-I_4). 
		\end{align*}
		Then, 
		\[
		u(x,y)=R(x,y)+iS(x,y)
		\]
		is the harmonic map that corresponds to $w$. 
			The domain of $R$ and $S$ is the largest possible open simply connected subset of $\mathbb{C}$ containing the origin so that the above expressions make sense.
	\end{proposition}
	\textbf{Remark.} \textit{Observe that the function $S$, by its definition preserves the sign of the initial data $S(0,0)$. Also, the image of $u$ is a subset of $\mathbb{H}^2$. However, these maps are not necessarily injective on their domain, as can be easily seen by Example \ref{ex FD}.}
	
	Therefore, given a pair of functions $(w,\theta)$ that satisfies the B{\"a}cklund transformation, there is an implicit formula for a harmonic map to the hyperbolic plane that involves the integrals $I_1, I_2, I_3, I_4$.  We provide an example of a harmonic map, first studied in \cite[Section 7]{FotDask}, that we recover by the algorithm of Proposition \ref{Theorem2}.

	We next characterize the solutions $w$ of the sinh-Gordon equation that are of the form
	\begin{equation}\label{separable}
		w(x,y)=2\arctanh({F(x)}{G(y)}),
	\end{equation}
	for some  non-constant functions $F$ and $G$. These functions, in general, turn out to be elliptic (only in certain cases they reduce to elementary functions).  We also prove that the initial value problem $\Delta w=2\sinh(2w)$, $\partial_{y}w(x,0)=0$, admits a solution of the form   (\ref{separable}), see Proposition \ref{Kenmotsu} for a class of such solutions. In this case,  we determine $\theta$ by considering a subclass of $(BT)$.
	\begin{definition}
			We say that the pair of functions $(w,\theta)$ is in the class $(BT_0)$,  when the pair is in $(BT)$ and  $\partial_{y}w(x,0)=0$,  $\theta(0,0)=\frac{\pi}{2}$.
		\end{definition}
	
	For this selection of $(w, \theta)$, and $w$ a solution of the form (\ref{separable}), we apply Proposition \ref{Theorem2} in order to prove our first main result,  which provides an explicit formula for a harmonic map that corresponds to $w$. This is a special case of the more general problem of finding all solutions of the sinh-Gordon equation, and the corresponding harmonic maps. Thus, it is a step (although small) towards the solution of the harmonic map problem. 
	
	Next, we state our first main result. 
	
	\begin{theorem}\label{new class} Assume that {$w(x,y)=2\arctanh({F(x)}{G(y)})$,  $b(x)=\frac{F^{\prime}(x)}{2F(x)}$} and $(w,\theta)\in (BT_0)$. 
		Define the functions $S$ and $R$ on the largest possible open simply connected subset of $\mathbb{C}$ containing the origin so that 
		\begin{equation*}
			S(x,y)=S(0,0) \frac{e^{2X(x) }(\sin{\theta(x,y)}+b(x))}{1+b(x)},
		\end{equation*}
		and 
		\begin{equation*}
			R(x,y)=R(0,0)+S(0,0)\frac{e^{2X(x)}\cos{\theta(x,y)}}{1+b(x)},
		\end{equation*}
		make sense, where $X=X(x)=\int_0^x \cosh w(t,0)dt$.
		Then  $u(x,y)=R(x,y)+iS(x,y)$ is the harmonic map that corresponds to $w$ .
	\end{theorem}

	Using the machinery of Theorem \ref{new class}, we construct an entirely new harmonic map to the hyperbolic plane. At this point, it is worth mentioning that the initial condition $\theta(0,0)=\frac{\pi}{2}$ was chosen just to facilitate computations; any other initial condition for $\theta$ at the origin would provide a corresponding harmonic map.
	
	In \cite{FotDask}, the approach pursued by the authors on the construction of harmonic maps, is to first consider a one-soliton (i.e., $w=w(x)$) solution of the sinh-Gordon equation and then solve the Beltrami equation (\ref{Beltrami eqn}) to find a harmonic map that corresponds to $w$. However, if $w(x,y)=2\arctanh({F(x)}{G(y)})$, and both $F, G$ are elliptic functions, then this Beltrami equation turns out to be equivalent to a Riccati equation. Therefore, explicit construction of harmonic maps by this approach is rather implausible. Instead, in place of the Beltrami equation, we use the B{\"a}cklund transformation combined with Theorem \ref{new class}, to provide an explicit formula for a harmonic map in terms of the functions $w$ and $\theta$.
	
	Next, instead of considering solutions to the sinh-Gordon equation first, we investigate the case of solutions to the sine-Gordon equation as a starting point. In particular, we prove that if $\theta=\theta(x)$ is a solution to the sine-Gordon equation (\ref{sin-Gordon eqn}), then the function $w$ determined by the B{\"a}cklund transformation is of the form (\ref{separable}).  We then provide an explicit formula for the corresponding harmonic map.	
	
	Our second main result is the following.
	
	{\begin{theorem}\label{thetaxthm}
			Let $\theta=\theta(x)$ be a solution to the sine-Gordon equation and let $w$ be the associated solution to the B{\"a}cklund transformation. Set $a = 2\cos\theta(0) + \theta'(0)$, $b = 2\cos\theta(0) - \theta'(0)$, and $k\in \mathbb{C}$  such that 
			$\tanh \frac{w(0,0)}{2}=- {\sqrt{\frac{a}{b}}\tan(k)}$. Then, the harmonic map $u=R+iS$ that corresponds to $w$ is given by
			\begin{align*} 
				R(x,y)&= R(0,0)+S(0,0)\cosh^2\frac{w(0,0)}{2} J\left(\frac{2\cos\theta(x) - \theta'(x)}{b}\right)\\
				&+ 2S(0,0)\frac{\sin\theta(x)(\cos(\sqrt{ab}y + 2k) - \cos(2k))}
				{2\cos\theta(0)\cos(2k) - \theta'(0)}, \\
				S(x,y)&=S(0,0)\frac{2\cos\theta(x)\cos(\sqrt{ab}y + 2k) - \theta'(x)}
				{2\cos\theta(0)\cos(2k) - \theta'(x)}, 
			\end{align*}
			where 
			\[ J(t)=\int_{1}^{t}
			\frac{u^{2} + \tanh^2\frac{w(0,0)}{2}}{u^{2}}
			\frac{bu^{2} + a}
			{\sqrt{2(8-ab)u^{2} - a^{2} - b^{2}u^{4}}}du.
			\]
			The domain of $R$ and $S$ is the largest possible open simply connected subset of $\mathbb{C}$ containing the origin so that the above expressions make sense.
		\end{theorem}
		
		As an application of Theorem \ref{thetaxthm}, we construct an entirely new example of a harmonic map from an open simply connected set to a subset of the hyperbolic plane, when the Hopf differential is non vanishing.
		
		Apart from the constructions \cite{ATW, CR, Wolf1} already mentioned, there are  many examples of harmonic diffeomorphisms, see for example \cite{C-T, Li-Ta2,S-T-W, Wolf1, Wolf2, X-A}. For instance, a minimal surface in $\mathbb{H}^2\times \mathbb{R}$ projects to a harmonic map to $\mathbb{H}^2$. The study of harmonic diffeomorphisms between Riemannian surfaces is central to the theory of harmonic maps (see for example \cite{S-Y, Hitchin} and observe that a twisted harmonic map, is a harmonic map in a local sense and that each harmonic map induces a solution of Hitchin self-duality equations). The most interesting case is when the surfaces are of constant curvature (see for example \cite{A-R,Kal,Li-Ta2,M2,Minsky, P-H,S-T-W,  Wolf1, Wolf2, Wolf3} and the references therein). In \cite{Han, Minsky, Wolf1, Wolf2, Wolf3}, the geometry of harmonic maps between hyperbolic surfaces is studied, while in \cite{M} the author proves a conjecture of R. Schoen on  harmonic diffeomorphisms between hyperbolic spaces.   For further results about the harmonic map equations see also  \cite{Han, H-T-T-W, Minsky, S-Y, Wolf1, Wolf2, Wolf3}. There is a similar analysis of the Wang equation, which is applied in the study of affine spheres in \cite{DW, Loftin, OT21} and \cite{Li19} explains the relation between the harmonic map and Wang equations with CMC surfaces and hyperbolic affine spheres in $\mathbb{R}^3$, while for connections with the anti-de Sitter geometry, see for instance \cite{Tam19} and \cite{BS}.
		
		Harmonic maps to surfaces of constant curvature are closely related to the elliptic sinh-Gordon equation. The sinh-Gordon  and sine-Gordon equations have many applications and they have both been the subject of extensive study, \cite{F-P,H}. Furthermore, the sinh-Gordon equation was  crucial to the breakthrough work \cite{Wente} on the Wente torus. There is a  close relation to the theory of constant mean curvature surfaces as well, \cite{Joa,K}.  For some striking relations of the sinh-Gordon equation with the CMC surfaces in Minkowski geometry we refer the reader to the papers \cite{BSS, BSS2}.
		
		The outline of this paper is as follows. In Section \ref{sec:Preliminaries} we  recall some basic notations and review certain known results. In Section \ref{sec:Harmonic maps to the upper half-plane model} we prove Proposition \ref{Theorem2} and in Section \ref{sec:The sinh-Gordon equation} we construct new solutions of the sinh-Gordon equation that are of the form $w(x,y)=2\arctanh({F(x)}{G(y)})$. In  Section \ref{sec:A new family of harmonic maps},  for this class of $w$, we find the explicit formula of a harmonic map that corresponds to $(w,\theta)\in (BT_0)$. Finally, in Section \ref{app}, we discuss one-soliton solutions of the sine-Gordon equation and the construction of the corresponding harmonic map.

		\section{Preliminaries}\label{sec:Preliminaries}
		
		In this section we discuss some necessary preliminaries.
		
		Let $u : M \rightarrow N$ be a map between
		Riemann surfaces $(M, g)$, $(N, h)$. The map $u$ is locally represented by
		$u = u(z, \bar{z}) = R + iS$, where $z=x+iy$. From now on, we use the standard notation:
		\[
		\partial_z=\frac{1}{2}(\partial_x-i\partial_y), \quad \partial_{\bar{z}}=\frac{1}{2}(\partial_x+i\partial_y).
		\]
		
		Recall that isothermal coordinates on a Riemannian manifold are local coordinates where the metric is conformal to the Euclidean metric. The existence of  isothermal coordinates on an arbitrary surface with a real analytic metric is a well-known fact, first proved by Gauss (see for instance \cite[Section 8, p.396]{J}). 
		Consider an isothermal coordinate system $(x, y)$ on $M$ such that
		\[g =  e^{f(z,\bar{z})}|dz|^2,\]
		where $z = x + iy$, and an isothermal coordinate system $(R, S)$ on
		$N$ such that
		\begin{equation*}\label{targetconf}h =  e^{F(u,\bar{u})}|du|^2,
		\end{equation*}
		where $u = R + iS$.
		The Gauss curvature on the target is given by the formula
		\[
		K_N(u, \bar{u})= -\frac{1}{2}\Delta F(u, \bar{u}) e^{-F(u, \bar{u})}.
		\]
		
		In isothermal coordinates,  a map between two surfaces is harmonic if and only if it satisfies
		\begin{equation}\label{harm isoth}
			\partial_{z\bar{z}}u + \partial_{u}F(u,\bar{u}) \partial_{z}u\partial_{\bar{z}}u=0,
		\end{equation}
		see \cite[Section 8, p.397]{J}. Note that this equation only depends on the conformal structure of $N$.
		
		The Hopf differential of $u$ is given by
		\begin{equation*}\label{eq:Hopf_theorem}
			\Lambda(z) dz^2=\left( e^{F(u,\bar{u})}\partial_{z} u \partial_{z} \bar{u}\right)dz^2.
		\end{equation*}
		It is well-known that if $u$ is holomoprhic or antiholomorphic then the zeros of $\partial_z u$ and $\partial_{\bar{z}} u$ are isolated of finite order, \cite[Proposition 2.1]{S-Y}. Due to the local nature of our study, \textit{we may assume from now on that $\Lambda$ does not vanish locally}.
		
We will make use of the following proposition.
		
			\begin{proposition}\label{prop:Hopf}
				A necessary and sufficient condition for a $C^2$ map $u$  with non vanishing Hopf differential, and an almost everywhere non vanishing Jacobian, to be a harmonic map, is that
				\begin{equation}\label{eq:Hopf}
					e^{F(u,\bar{u})} \partial_z u    \partial_z\bar{u}= e^{-\mu (z)},
				\end{equation}
				where $\mu (z)$ is a holomorphic function. 
		\end{proposition}
		\begin{proof}
			We sketch the proof for the reader's convenience.
				Assume first that $u$ is a harmonic map. Then the harmonic map equation (\ref{harm isoth}) is equivalent to 
				\[
				\partial_{\bar{z}}(e^{F(u,\bar{u})}  \partial_z u \,\partial_z \bar{u})= 0.
				\]
				Then, the claim  follows under the assumption that $u$ has non vanishing Hopf differential.

			For the converse direction, let  $e^{-\mu(z)}$ be holomorphic. Then, observe that  
				\[
				\text{if} \quad e^{F(u,\bar{u})}  \partial_z u \, \partial_z{\bar{u}}= e^{-\mu(z)} \quad 
				\mbox{then} \quad  
				e^{F(u,\bar{u})} \partial_{\bar{z}}u\, \partial_{\bar{z}}{\bar{u}}=e^{-\overline{\mu(z)}}.
				\]
				Thus, differentiating these equations in terms of $\bar{z},z$ respectively, we find 
				\[ \partial_z \bar{u} \;(\partial^2_{z\bar{z}}u + \partial_uF\, \partial_zu \, \partial_{\bar{z}}u) +
				\partial_{z}u \;(\partial^2_{z\bar{z}}\bar{u} +  \partial_{\bar{u}}F \,\partial_{z}\bar{u}\, \partial_{\bar{z}}\bar{u})=0,\]
				and
				\[ \partial_{\bar{z}} \bar{u} \;(\partial^2_{z\bar{z}}u + \partial_uF\, \partial_zu \, \partial_{\bar{z}}u) +
				\partial_{\bar{z}}u \;(\partial^2_{z\bar{z}}\bar{u} +  \partial_{\bar{u}}F \,\partial_{z}\bar{u}\, \partial_{\bar{z}}\bar{u})=0.\]

			Taking into account that the Jacobian 
				\[\|u_z\|^2-\|u_{\bar{z}}\|^2 =\partial_{x}\,R\partial_{y}S-\partial_{x}\,S\partial_{y}R\] is non vanishing almost everywhere, we deduce that the harmonic map equation hold true almost everywhere. Then, under the $C^2$ assumption on $u$ we deduce that $u$ is harmonic everywhere. 
			
		\end{proof}
		
		A fundamental result in the theory of harmonic maps, which connects the harmonic map problem with the sinh-Gordon equation, is the following. \begin{proposition}[\cite{Minsky}, \cite{Wolf1}] \label{PropMisnkyWolf}
			Let $u:M\rightarrow N$ be a harmonic map. Then, it satisfies the Beltrami
			equation
			\begin{equation*}\label{Beltrami eqnOLD}
				\frac{\partial_{\bar{z}}u}{\partial_{z}u}=e^{-2w+i\phi},
			\end{equation*}
			and $\phi$ is a harmonic function, i.e. $\partial^2_{z\bar{z}}\phi=0$. Furthermore, if $\psi$ is the conjugate harmonic function to $\phi$, then
			\begin{equation*}
				K_N=-\frac{2 \partial^2_{z\bar{z}}w  }{\sinh 2w}e^{\psi},
			\end{equation*}
			where $K_N$ is the curvature of the target manifold $N$.
		\end{proposition}
		
		On the domain, we can choose a specific coordinate system in order to considerably facilitate the calculations. This \textit{specific}  system is defined by the conformal transformation 
		\begin{equation*}\label{eq:specific}
			Z= \int e^{-\mu(z)/2} \, dz.
		\end{equation*}  
		This transformation is well defined since we work on a simply connected subset and we can assume that the Hopf differential is not vanishing since its zeroes are isolated.
		In this specific system the above equations simplify by considering $\mu(z)=0$. In other words, under the assumption that the Hopf differential does not vanish, we may assume that $\mu=0$.
		In particular, the harmonic map condition (\ref{eq:Hopf}) becomes \begin{equation}\label{Hopfnew}
			e^{F(u, \bar{u})}\partial_{z} u\partial_{z}\bar{u}=1.
		\end{equation}

		It follows from Proposition \ref{PropMisnkyWolf} that the harmonic map $u$  determined by (\ref{Hopfnew}) satisfies the Beltrami equation (\ref{Beltrami eqn}), where in turn the function $ w(z, \bar{z})$  occurring in the Beltrami coefficient} satisfies the sinh-Gordon equation (\ref{sinh-Gordon eqn}).  Conversely, as we have already mentioned, if $w$ is a solution of the sinh-Gordon equation (\ref{sinh-Gordon eqn}) and  $u$ is a solution of the Beltrami equation (\ref{Beltrami eqn}), then, $u$ is a harmonic map,  \cite{FotDask}.
	As a result, there is a classification of harmonic diffeomorphisms via the classification of the solutions of the sinh-Gordon equation. For further details about the association of the harmonic map problem to the sinh-Gordon equation, we refer to \cite[Section 3]{FotDask}.
	From now on, we   use the aforementioned specific coordinate system.
	
	We assume that the target is the upper half hyperbolic plane
	\[
	\mathbb{H}^{2}=\{R+iS: \; R\in \mathbb{R},\; S>0\}, 
	\] equipped with the hyperbolic metric 
	\[h=\frac{dR^2+dS^2}{S^2}\]
	of curvature $-1$.
	As mentioned earlier, in the specific coordinate system, the map $u=R+iS$ is harmonic if
	(\ref{eq:HarmHyper}) holds true, or equivalently if the following system is satisfied:
	\begin{equation}\label{Cartesian1}
		\partial_{x}R \;\partial_{y}R + \partial_{x}S\;\partial_{y}S = 0
	\end{equation}
	\begin{equation}\label{Cartesian2}
		\frac{(\partial_{x}R)^{2} + (\partial_{x}S)^{2} - (\partial_{y}R)^{2} - (\partial_{y}S)^{2}}
		{S^{2}} = 4.
	\end{equation}
	
	Finally, as we have described, our approach has as its starting point the B{\"a}cklund transformation (\ref{Back1intro})-(\ref{Back2intro}). For computational reasons, let us point out that it is more convenient to verify that a pair $(w, \theta)$ is a solution if the following system holds:
	\begin{align*}
		\partial_x W\;(1+\Theta^2) - \partial_y\Theta\;(1-W^2)&=-4W\Theta,\\
		\partial_y W\;(1+\Theta^2) + \partial_x\Theta\;(1-W^2)&=-(1+W^2)(1-\Theta^2),
	\end{align*}
	where $W=\tanh\frac{w}{2}$ and $\Theta=\tan\frac{\theta}{2}$.
	
	\section{Proof of Proposition \ref{Theorem2}}\label{sec:Harmonic maps to the upper half-plane model}

	In this section we consider $(w,\theta)\in (BT)$.  In other words, $w$ and $\theta$ are related by the B{\"a}cklund transformation.  Our aim is to construct a harmonic map that corresponds to $w$. As usual, the target is assumed to be the upper half-plane equipped with the hyperbolic metric.
	We first observe that the following result holds true. 
		\begin{lemma}Assume that $(w,\theta)\in (BT)$. Then, there exist functions $R$ and $S$ such that
			\begin{align}
				\partial_xS&=2S\cosh w\sin \theta \label{Sx}\\
				\partial_yS&=2S\sinh w\cos \theta \label{Sy}\\
				\partial_xR&=2S\cosh w\cos \theta \label{Rx}\\
				\partial_yR&=-2S\sinh w\sin \theta \label{Ry}.
			\end{align} 
		\end{lemma}
		\begin{proof}	The existence of the functions $R$ and $S$ is ensured by the B{\"a}cklund transformation of $w$ and $\theta$. Indeed, the compatibility conditions $$\partial_{xy}^2 S=\partial_{yx}^2 S,\quad  \partial_{xy}^2 R=\partial_{yx}^2 R$$ 
			follow by  the B{\"a}cklund transformation equations (\ref{Back1intro})-(\ref{Back2intro}).
		\end{proof} 
		\noindent	\textbf{Remark.} The Jacobian of the map $u$ is $J(u)=2S^2 \sinh{2w}$.
	
	Our next goal is to prove the following result, which plays a key role for the proof of Proposition \ref{Theorem2}.
	
	\begin{lemma}Assume that $(w,\theta)\in (BT)$ and let $(R,S)$ satisfy the system of equations (\ref{Sx})-(\ref{Ry}).
		Then, 
		\[
		u(x,y)=R(x,y)+iS(x,y)
		\]
		is the harmonic map to the hyperbolic plane that corresponds to $w$.
	\end{lemma}
	\begin{proof}
		
		Consider \[
		u=R+iS,
		\]
		and observe that
		\[
		\partial_zu=\frac{1}{2}((\partial_xR+\partial_yS)+i(\partial_xS-\partial_yR)).
		\]
		Then (\ref{Sx})-(\ref{Ry}), imply that
		\begin{equation}\label{uz}
			\partial_zu=S(\cosh w+\sinh w)(\cos\theta+i\sin\theta)=Se{^w} e^{i\theta}.
		\end{equation}
		Similarly, we compute
		\begin{equation}\label{uzbar}
			\partial_{\bar{z}}u=Se^{-w} e^{i\theta}.\end{equation}
		By (\ref{uz}) and (\ref{uzbar}), it follows that
		\[
		\frac{\partial_{\bar{z}}u}{\partial_zu}=e^{-2w}.
		\]
		Furthermore, by (\ref{Hopfnew}), the conformal factor in the target metric is equal to \[e^{F(R,S)}=\frac{1}{\partial_zu\partial_z\bar{u}}=\frac{1}{S^2}.\]
		Thus, $u$ is a harmonic map to the hyperbolic upper half hyperbolic plane that corresponds to $w$.
	\end{proof}

	Given now $w$ and $\theta$, we   find the associated harmonic map $u=R+iS$. More precisely, we provide the implicit formulas for $R$ and $S$ that involve the integrals $I_1, I_2, I_3, I_4$.

	\begin{proof}[End of the proof of Proposition \ref{Theorem2}]
		{{Our task is now to} solve {the }system (\ref{Sx})-(\ref{Ry}). {We} begin with (\ref{Sx})} for $y=0$, which yields
		\[
		\frac{\partial_xS}{S}(x,0)=2\cosh w(x,0)\sin\theta(x,0).
		\]
		Integrating, we obtain 
		\begin{equation*}\label{S(x,0)}
			S(x,0)=S(0,0)e^{2\int_0^x \cosh w(t,0)\sin\theta(t,0)dt}.
		\end{equation*}
		So, (\ref{Sy}) implies that
		\begin{equation*}
			S(x,y)=S(0,0)e^{2\int_0^x \cosh w(t,0)\sin\theta(t,0)dt+2\int_0^y \sinh w(x,s)\cos\theta(x,s)ds}.
		\end{equation*}
		Thus
		\begin{align*}
			S(x,y)=S(0,0)e^{2(I_1+I_2)}. 
		\end{align*}
		
		In order now to compute $R(x,y)$, we consider (\ref{Rx}) for $y=0$, which yields
		\begin{equation*}\label{R(x,0)}
			R(x,0)=R(0,0)+2\int_0^xS(t,0)\cosh w(t,0)\cos\theta(t,0)dt.
		\end{equation*}
		So, (\ref{Ry}) implies that
		\begin{align*}\
			R(x,y)=R(0,0)&+2\int_0^xS(t,0)\cosh w(t,0)\cos\theta(t,0)dt\notag\\&-2\int_0^yS(x,s)\sinh w(x,s)\sin\theta(x,s)ds, 
		\end{align*}
		thus
		\begin{align*}
			R(x,y)=R(0,0)+2S(0,0)(I_3-I_4). 
		\end{align*}
	\end{proof}
	
	We now provide a concrete example of a harmonic map to show how the algorithm in Proposition \ref{Theorem2} can be implemented. In particular, we recover the harmonic map obtained in \cite[Section 7]{FotDask}.
	\begin{example}\label{ex FD}
		Suppose that $\tanh{\frac{w(x,y)}{2}}=\frac{2y}{\cosh{2x}}$ and consider the case $R(0,0)=0$, $S(0,0)=-\frac{1}{4}$ as in {\cite[p.22]{FotDask}}.  Using the B{\"a}cklund transformation we find that 
		\[
		\tan\frac{\theta(x,y)}{2}=\coth{x}.
		\] 
		
		Then, a lengthy calculation yields
		\begin{align*}
			I_1 &= \frac{1}{2} \log{\cosh{2x}},\\
			I_{2}&=\frac{1}{2}\log{\left( 1-\frac{4 y^2}{\cosh^2{2x}}\right)},\\
			I_3 &= -x,\\
			I_4 &=2y^2 \tanh{2x},
		\end{align*}
		{and} therefore,
		\begin{align*}
			R&= \frac{x}{2}+y^2 \tanh{2x} \\
			S&= \frac{y^2}{\cosh{2x}}-\frac{\cosh{2x}}{4}.
		\end{align*}
		One can verify by (\ref{Cartesian1})-(\ref{Cartesian2}) that $u=R+iS$ is a harmonic map to the hyperbolic plane that corresponds to $w$. 
		
		The domain of definition of $u$ is the set \[\Omega=\left\{ (x,y): \;  \left| y\right|<\frac{1}{2}\cosh{2x}\right\}\] where the map is a well defined $C^{2}$ map whose Jacobian is almost everywhere non vanishing. 
	\end{example}

	\section{The sinh-Gordon equation}\label{sec:The sinh-Gordon equation}
	
	In this section we provide a new family of solutions of the sinh-Gordon equation.    
	
	The motivation for obtaining solutions of the sinh-Gordon equation is their correspondence to harmonic maps between surfaces, as described in Section \ref{sec:Preliminaries}.  In addition, it is exactly the form of the special solutions to the sinh-Gordon equation considered in this section, that will allow explicit expressions of harmonic maps to hyperbolic plane. We therefore postpone our main goal, the construction of harmonic maps, to the next section. 
	
	In \cite[Section 5]{FotDask}, the authors construct harmonic maps when $w$ is a one-soliton solution to (\ref{sinh-Gordon eqn}). In some sense, they obtain solutions to this sinh-Gordon equation that are of the form 
	\[
	w(x,y)=2\arctanh({F(x)}).
	\]
	In what follows,  we focus on solutions to the sinh-Gordon equation (\ref{sinh-Gordon eqn}) that are of the form (\ref{separable}), i.e. $w(x,y)=2\arctanh({F(x)}{G(y)})$.
	Without loss of generality, we assume that $F$ and $G$ are non-constant functions. 
	
	Notice that condition  (\ref{separable}) is equivalent to $b=\frac{\partial_{x}w(x,y)}{2\sinh{w(x,y)}}=b(x),$ where $b(x)=\frac{F'(x)}{2F(x)}$. 
	One can prove that (\ref{separable}) is also equivalent to the differential equation
		\begin{equation}\label{sep2}
			\partial_{xy}^{2}w(x,y)=\partial_{x}w(x,y)\partial_{y}w(x,y)\coth{w(x,y)}.
		\end{equation}
		Indeed, for the one direction, starting from $\tanh\frac{w}{2}=FG$ we get $\partial_x w=\frac{F'}{F}\sinh w$ and $\partial_y w=\frac{G'}{G}\sinh w$. Thus, differentiating $w_x$ in $y$ and using $w_y$, we get (\ref{sep2}). Conversely, integrating $\frac{\partial^2_{xy}w}{\partial_x w}=\partial_y w \coth w$ in $y$ yields $\frac{\partial_x w}{\sinh w}=b(x)$. Integrating in $x$, we get for some $c(y)$ that $\log\left| \tanh\frac{w}{2} \right|=\int b(x)dx+c(y)$, whence the claim follows.
	In \cite{Abr}, the authors were lead to this equation for geometric reasons related to CMC surfaces. It is an open problem to relate this equation with the geometry of harmonic maps. 
	
	The basic idea is that under condition (\ref{separable}), we can solve this sinh-Gordon equation by a separation of variables argument. Then, the problem is  reduced to solving certain non-linear ODEs, which  in turn can be solved by the use of elliptic functions.
	
	We now prove the following result. 
	\begin{proposition}\label{w thm}
		{If $w$ is} a solution of the sinh-Gordon equation (\ref{sinh-Gordon eqn})
		of the form 
		\begin{equation*}
			w(x,y)=2\arctanh({F(x)}{G(y)}),
		\end{equation*}
		then the functions $F,G$ satisfy the differential equations
		\begin{align}
			(F^{\prime}(x))^2&=AF^4(x)+BF^2(x)+C \label{odef}\\
			(G^{\prime}(y))^2&=-CG^4(y)-(B-4)G^2(y)-A \label{odeg}, 
		\end{align}
		where $A,B,C$ are arbitrary constants.
	\end{proposition}
	
	\begin{proof}
		Set $G(y)=\frac{1}{H(y)}$.	 Writing $w=2\arctanh\frac{F(x)}{H(y)}$, we find from {(\ref{sinh-Gordon eqn})} that
		\begin{align}
			&\;\frac{F^{\prime\prime}(x)}{F(x)}\left({H^2(y)}
			-{F^2(x)}\right)+2F^{\prime}(x)^2 
			-\frac{H^{\prime\prime}(y)H(y)-2H^{\prime}(y)^2}{H(y)^2}\times \notag \\
			&\times\left({H^2(y)}-{F^2(x)}\right)+2F^2(x)\frac{H^{\prime}(y)^2}{H^2(y)} \label{eq:sinhSOL} \\&=4\left(F^2(x)+H^2(y)\right).\notag
		\end{align}
		Differentiating with respect to $x$ and $y$, we obtain
		\[
		2H(y)H^{\prime}(y)\left( \frac{F^{\prime\prime}(x)}{F(x)} \right)^{\prime}+2F(x)F^{\prime}(x)\left( \frac{H^{\prime\prime}(y)}{H(y)} \right)^{\prime}=0,
		\]
		or
		\begin{equation*}\label{7}
			\frac{1}{F(x)F^{\prime}(x)}\left( \frac{F^{\prime\prime}(x)}{F(x)} \right)^{\prime}=-\frac{1}{H(y)H^{\prime}(y)}\left( \frac{H^{\prime\prime}(y)}{H(y)} \right)^{\prime}=4A=\text{constant}.
		\end{equation*}
		Therefore, 
		\begin{align*}
			\left( \frac{F^{\prime\prime}(x)}{F(x)} \right)^{\prime}=4AF(x)F^{\prime}(x), \quad\left( \frac{H^{\prime\prime}(y)}{H(x)} \right)^{\prime}=4AH(y)H^{\prime}(y).
		\end{align*}
		Integrating twice and taking into account (\ref{eq:sinhSOL}), we obtain
		\begin{equation*}\label{8}
			F^{\prime}(x)^2=AF^4(x)+BF^2(x)+C,
		\end{equation*}
		and	
		\begin{equation*}
			H^{\prime}(y)^2=-AH^4(y)-(B-4)H^2(y)-C.
		\end{equation*}
		Finally, we find
		\[
		G^{\prime}(y)^2=-CG^4(y)-(B-4)G^2(y)-A.
		\]
	\end{proof}

	An argument similar to the one used in \cite{K}, reveals that the following result, which is a special case of Proposition \ref{w thm}, holds true.
	\begin{proposition}\label{Kenmotsu}
		Let $w_0>0$, and $\alpha,\beta>0$, such that
		\[
		\alpha+\beta=\cosh {w_0}>1.
		\]
		Consider $f(x)$, $g(y)$ such that
		\begin{align*}
			(f')^2&=f^4-4\left(1+\alpha^2-\beta^2\right)f^2+4^2  \alpha^2  
			\\
			(g')^2&=g^4-4\left(1+\beta^2-\alpha^2\right)g^2+4^2  \beta^2,
		\end{align*}
		with $f(0)=0$, $f'(0)=-4\alpha$, $g(0)=0$, $g'(0)=-4 \beta$.
		Then the function $w(x,y)$  given by
		\begin{equation*}
			\tanh\frac{w(x,y)}{2}=\tanh\frac{w_0}{2}e^{-\int_0^xf(t)dt}e^{-\int_0^yg(s)ds},
		\end{equation*}
		is such that  $\Delta w=2\sinh(2w)$ and $w(0,0)=w_0$.
	\end{proposition}

	Finally, note that in general, $F$ and $G$ in Proposition \ref{w thm}, and $f$ and $g$ in Proposition \ref{Kenmotsu} respectively, are given in terms of elliptic functions, but reduce to elementary functions for certain choices of the coefficients of the ODEs.  In particular, we observe that the only case where one of the functions $f$, $g$ is elementary is when $|\alpha-\beta|=1$. In this case actually both functions $f$, $g$ end up to be elementary.

	\section{A new family of harmonic maps}\label{sec:A new family of harmonic maps}\
	
	In this section we  assume that $(w,\theta)\in (BT_0)$  and $w$ is of the form (\ref{separable}). In other words,  the functions $w=2\arctanh(F(x)G(y))$ and $\theta$ are related by the B{\"a}cklund transformation and $\partial_{y}w(x,0)=0$, $\theta(0,0)=\frac{\pi}{2}$. We start with a given solution $w$ of the sinh-Gordon equation and we find a formula for $\theta$ using the B{\"a}cklund transformation. We finally find a formula for the harmonic map $u$ that corresponds to $w$. 
	
	\subsection{B{\"a}cklund transformation} \label{sec:Backlund transform of the sinh-Gordon equation}
	
	Our first task is to solve the system (\ref{Back1intro})-(\ref{Back2intro}). Suppose that $w$ is a solution of the sinh-Gordon equation, of the form
	\begin{equation}\label{omega1}
		\tanh \frac{w}{2}=F(x)G(y).
	\end{equation}
	
	Define 
	\begin{align}
		X&=X(x)=\int_0^x \cosh w(t,0)dt=\int_0^x \frac{1+F^2(t)G^{2}(0)}{1-F^2(t)G^{2}(0)}dt, \label{Xvar}\\
		Y&=Y(x,y)=\int_0^y \sinh w(x,s)ds=\int_0^y \frac{2F(x)G(s)}{1-F^2(x)G^2(s)}ds\label{Yvar}.
	\end{align}
	
	The new variable $Y$ appears in the following lemma, which serves as an intermediate step for the formula of the harmonic map in Theorem \ref{new class}, where the other variable $X$ appears. We claim that we can find explicit expressions of $X$ and $Y$ in terms of $x$ and $y$. Assuming for the moment that this claim is true, we complete the construction of $\theta$ as follows, in the special case when {$\partial_{y}w(x,0)=0,$ $\theta(0,0)=\frac{\pi}{2}$.}

	\begin{lemma}\label{BackLemma}
		If $	w(x,y)=2\arctanh({F(x)}{G(y)})$, $b(x)=\frac{F^{\prime}(x)}{2F(x)}$ and $(w,\theta)\in (BT_0)$, then 
		\begin{equation*}
			\tan\frac{\theta(x,y)}{2}=\frac{1}{b(x)}\left( \sqrt{b^2(x)-1}\tan({J_1(x)+J_2(x,y)}) -1 \right),%\label{tanthetaXY}
		\end{equation*}
		where{
			\begin{align*}
				J_1=J_1(x)=\arctan \left( \frac{b(x)+1}{\sqrt{b^2(x)-1}}  \right),
		\end{align*}}
		and
		\begin{align*}
			J_2=J_2(x,y)=\sqrt{b^2(x)-1}\;Y(x,y).
		\end{align*}
	\end{lemma} 
	
	\begin{proof}
		By assumption, we have 
		\begin{equation}\label{wy0}
			\partial_y w(x,0)=0.
		\end{equation}
		Taking (\ref{Back2intro}) for $y=0$ and applying (\ref{wy0}), gives
		\begin{equation*}
			\partial_x\theta(x,0)=-2\cosh w(x,0)\cos\theta(x,0).
		\end{equation*}
		Given the initial condition $\theta(0,0)=\frac{\pi}{2}$, we can take
		\begin{equation*}\label{tanthetaX0}
			\theta(x,0)=\frac{\pi}{2}.
		\end{equation*}
		Next, observe that since $w(x,y)=2\arctanh({F(x)}{G(y)})$, we have
		\[
		\partial_x w(x,y)=\frac{F^{\prime}(x)}{F(x)}\sinh w(x,y),
		\]	
		hence (\ref{Back1intro}) implies
		\begin{equation}\label{thetay}
			\partial_y\theta(x,y)=2\sinh w(x,y)\left(\sin\theta(x,y)+\frac{F^{\prime}(x)}{2F(x)}\right).
		\end{equation}
		An integration yields
		\begin{equation*}
			\tan\frac{\theta(x,y)}{2}=\frac{1}{b(x)}\left( \sqrt{b^2(x)-1}\tan({J_1+J_2}) -1 \right),
		\end{equation*}
		where
		\begin{align*}
			J_1=J_1(x)&=\arctan \left( \frac{1+b(x)}{\sqrt{b^2(x)-1}}  \right),
		\end{align*}
		and
		\begin{align*}
			J_2&=J_2(x,y)=\sqrt{b^2(x)-1}\;Y(x,y).
		\end{align*}
	\end{proof}
	Note that the only, but crucial use of the condition (\ref{omega1}) is that $\frac{\partial_{x}w(x,y)}{2\sinh{w(x,y)}}=\frac{F'(x)}{2F(x)}=b(x)$ is a function of $x$, which facilitates the integration in (\ref{thetay}).
	
	It remains  to find explicit formulas of $X$ and $Y$ in terms of $x$ and $y$. Recall that by (\ref{Xvar}) and (\ref{Yvar}), these expressions are given by
	\[
	X=\int_{0}^{x}\frac{1+F^{2}(t)G^2(0)}{1-F^{2}(t)G^2(0)}dt, 
	\]
	and
	\[
	Y=2F(x)\int_{0}^{y} \frac{G(s)}{1-F^2(x) G^2(s)}ds, 
	\]
	where $F$ and $G$ are determined by (\ref{odef}) and (\ref{odeg}), respectively. Therefore, $F$ and $G$ are elliptic functions in general, making the integrals in the computation of $X$ and $Y$ nontrivial. In order to compute these integrals,  we make use of (\ref{odef}) and (\ref{odeg}) respectively.
	
	Write
	\begin{align*}
		X&=\int_{0}^{x}\frac{1+F^{2}(t)G^2(0)}{1-F^{2}(t)G^2(0)}dt\\
		&=\int_{0}^{x}\frac{\left(1+F^{2}(t)G^2(0)\right) F^{\prime}(t)}{\left(1-F^{2}(t)G^2(0)\right)\sqrt{A F^{4}(t)+B F^2(t)+C}}dt.
	\end{align*}
	Hence,
	\begin{equation*}
		X=\int_{F(0)}^{F(x)}\frac{1+v^{2}G^2(0)}{\left(1-v^{2}G^2(0)\right)\sqrt{A v^{4}+B v^2+C}}dv.
	\end{equation*}
	This integral can be computed explicitly and it involves first and third kind elliptic functions.
	
	Similarly, we write
	\[
	Y=F(x)\int_{0}^{y} \frac{2G(s) G^{\prime}(s)}{\left(1-F^2(x) G^2(s)\right)\sqrt{-CG^4(s)-(B-4)G^2(s)-A }}ds.
	\]
	Hence,
	\begin{equation*}
		Y=F(x)\int_{G^2(0)}^{G^2(y)}\frac{du}{\left(1-F^2(x) u\right)\sqrt{-Cu^2-(B-4)u-A }}.
	\end{equation*}
	This integral can be computed explicitly and it is an elementary function of $G^2(y)$. 
	
	\subsection{The corresponding harmonic maps.}
	Given Lemma \ref{BackLemma}  we  now complete the proof of our main result.

	\begin{proof}[End of the proof of Theorem \ref{new class}]
		Note that from Lemma \ref{BackLemma} we know the formula of $\theta$. Thus, it is left to express the harmonic map in terms of $\theta$ and $b=\frac{\partial_{x}w(x,y)}{2\sinh{w(x,y)}}=b(x)$. The proof relies on the formulas given by Proposition \ref{Theorem2}. The trick of the proof is to use (\ref{thetay}) in order to substitute\begin{equation*}
			\sinh{w(x,s)}=\frac{\partial_{s}\theta(x,s)}{2 (\sin{\theta(x,s)}+b(x))},
		\end{equation*} 
		in the integrals $I_2$ and $I_4$. For $I_1$ and $I_3$, we use the initial condition $\theta(x,0)=\frac{\pi}{2}$. A direct calculation now yields
		\begin{align*}I_1&=X, \\
			I_2&=\frac{1}{2}\log{\frac{\sin{\theta(x,y)}+b}{1+b}},\\
			I_3&=0,\\
			I_4&=-\frac{1}{2}\frac{e^{2X}\cos{\theta(x,y)}}{1+b}.
		\end{align*}
		Then the result follows by Prosopition \ref{Theorem2}  and a straightforward computation of $R$ and $S$.
	\end{proof}

	We elaborate on this result by an example motivated by Proposition \ref{Kenmotsu}. We consider the case $\alpha=1, \beta=2$, $\tanh{\frac{w(0,0)}{2}=\frac{\sqrt{2}}{2}}$ in Proposition \ref{Kenmotsu}, which implies that both $f, g$ are elementary functions. Next, using Theorem \ref{new class} we provide an entirely new harmonic map to the hyperbolic plane.
	
	\begin{example}
		Consider $S(0,0)=1$, $R(0,0)=0$ and $\tanh{\frac{w(x,y)}{2}}=\frac{\sqrt{2}}{2}\frac{\cosh{(2\sqrt{2}y)}}{\cos{(2x)}}$.
		We find 
		\[
		X=x+\arctanh{(\tan{(2x)})},\quad  b=b(x)=\tan{(2x)},
		\]
		\[
		Y= {\frac{1}{\sqrt{1-\tan^2(2x)}} }  \arctanh{\frac{\sinh{(2\sqrt{2}y)}}{\sqrt{1-2\sin^2{(2x)}}}},
		\]
		and
		\[
		\tan{\frac{\theta(x,y)}{2}}=\frac{\sinh{(2\sqrt{2}y)}+\cos{(2x)}-\sin{(2x)}}{\cos{(2x)}-\sin{(2x)}-\sinh{(2\sqrt{2}y)}}.
		\]
		After a lengthy computation we obtain
		\[{
			S(x,y)=e^{2x}\frac{1+2 \cos{(4x)}-\cosh{(4\sqrt{2}y)}}{1+\cosh{(4\sqrt{2}y})-2\sin{(4x})}}
		\]
		and
		\[{
			R(x,y)=-4e^{2x}\frac{\cos{(2x)}\sinh{(2\sqrt{2}y)}}{1+\cosh{(4\sqrt{2}y})-2\sin{(4x})}.}
		\]
		Therefore, $u=R+iS$ is the harmonic map that corresponds to $w$.
		
		The domain of definition of $u$ is the set
			\[
			\Omega=\left\{ (x,y): \; \left|\frac{\sqrt{2}}{2}\frac{\cosh{(2\sqrt{2}y)}}{\cos{(2x)}}\right|<1 \right\}
			\] where the map is a well defined $C^{2}$ map whose Jacobian is almost everywhere non vanishing. 
	\end{example}

	\section{One-soliton solutions of the Sine-Gordon equation}\label{app}
	In the previous section, we started with a special solution $w$ to the sinh-Gordon equation and then determined $\theta$ by the B{\"a}cklund transformation, in order to construct harmonic maps by the algorithm proposed in Proposition \ref{Theorem2}. We now follow the reverse direction, and start with a special solution $\theta$ to the sine-Gordon equation. 
	
	More precisely, we are motivated by an example in \cite[p.18]{FotDask}, where a special solution $\theta=\theta(x)=\arcsin \tanh (2x)$ of the sine-Gordon equation is considered. Notice that in the aforementioned example, $\theta(x)$ is an elementary function. However, we prove in this section that in general $\theta(x)$ is an elliptic function. In addition, we prove that if $\theta=\theta(x)$ is a solution to the sine-Gordon equation (\ref{sin-Gordon eqn}), then the function $w$ determined by the B{\"a}cklund transformation is in the family of solutions that we have studied in Section \ref{sec:The sinh-Gordon equation},  i.e. $w=\arctanh (F(x)G(y))$. Next, we provide an  explicit formula for the corresponding harmonic map. 
	
	For clarity reasons, we first present the definitions and properties of the elliptic Jacobi functions. The formulation used in this paper is taken from \cite{MilThom64}.

	The elliptic integral of the first kind $F(\phi \vert n)$  and the Jacobi elliptic function $\sn(v\vert n)$  are defined by the formula
	\begin{equation*}\label{eq:Jacobi_ampli}
		F(\phi\vert n) = v=
		\int_0^x \, \dfrac{dt}{   \sqrt{ \left( 1-t^2\right)\left( 1 -n t^2\right)}}=
		\sn^{-1}(x\vert n),
	\end{equation*}
	where
	\[
	\sn(v \vert n)=x=\sin \phi.
	\]
	{Define}
	\[\cn(v \vert n)=\cos \phi ,
	\quad \dn(v \vert n)= \sqrt{1-n  \sin^2\, \phi}.
	\]
	
	We first prove the following result, which provides the characterization for $w$.
	
	\begin{proposition}
		If $\theta = \theta(x)$ is a solution of the sine-Gordon equation
		\begin{equation*}
			\Delta\theta = - 2\sin(2\theta),
		\end{equation*}
		then the associated solution $w$ of the sinh-Gordon equation is 
		\begin{equation*}
			\tanh \frac{w(x,y)}{2} = F(x)G(y) = 
			\frac{\sqrt{ab}\tan(\frac{\sqrt{ab}}{2}y + k)}{\theta'(x) - 2\cos\theta(x)},
		\end{equation*}
		where $a = 2\cos\theta(0) + \theta'(0)$, $b = 2\cos\theta(0) - \theta'(0)$, and $k\in \mathbb{C}$ is such that 
		$\tanh \frac{w(0,0)}{2}=- {\sqrt{\frac{a}{b}}\tan(k)}$.
	\end{proposition}
	
	\noindent	\textbf{Remark.} The above formula holds true also for the case $a b <0.$

	\begin{proof}
		Our first task is to determine the function $\theta$. Consider $\theta= \theta(x)$ to be a solution of the sine-Gordon equation
		$$
		\Delta\theta = -2\sin(2\theta).
		$$
		Then,
		$$
		\theta'' = -2\sin(2\theta) 
		$$
		and therefore 
		$$
		\theta'(x)^{2} = - 4\sin^{2}\theta(x) + c^{2},
		$$
		where $c^2 = \theta'(0)^{2} + 4\sin^2\theta(0)$.
		A simple calculation reveals that
		$$
		\int_{0}^{\theta(x)}\frac{d\psi}{\sqrt{1 - \frac{4}{c^{2}}\sin^2(\psi)}} = c_{1} + c x,
		$$
		which implies
		\begin{align}\label{sn-theta}
			\sin\theta(x) = \sn(c x + c_{1}| \frac{4}{c^{2}}), \; \text{ where } c_{1} = \int_{0}^{\theta (0)}\frac{d\psi}{\sqrt{1 - \frac{4}{c^{2}}\sin^2(\psi)}}.
		\end{align}
		Moreover, we have 
		\begin{align}\label{cndn-theta}
			\cos\theta(x) = \cn(c x + c_{1}| \frac{4}{c^{2}}),\quad \theta'(x)=c\; \dn(c x + c_{1}| \frac{4}{c^{2}}), 
		\end{align}
		from which we derive the following differential equation 
		\begin{equation}\label{thetaODE}
			\theta'^2-4\cos^2 \theta=c^2-4.
		\end{equation}
		Note that this implies
		\begin{equation*}
			ab=4-c^2.
		\end{equation*}
		
		We next prove that the function $w$ is of the form $w=2\arctanh (F(x)G(y))$. For that, we turn to the B{\"a}cklund transformation, which rewrites as 
		\begin{align}
			\partial_x w&=-2\sinh w \sin \theta, \label{Back1Sec5} \\
			\partial_y w+\theta'(x)&=-2\cosh w \cos \theta \label{Back2Sec5}.
		\end{align}
		Integrating \eqref{Back1Sec5}, we get
		$$
		\int_{w(0,y)}^{w(x,y)} \frac{d\phi}{\sinh(\phi)} + 2\int_{0}^{x}\sin\theta(s)ds = 0.
		$$
		It follows that
		$$
		\tanh(\frac{w(x,y)}{2}) = \tanh(\frac{w(0,y)}{2})e^{-2I(x)},
		$$
		where $I(x) = \int_{0}^{x}\sin\theta(s)ds$.
		
		We now apply Theorem \ref{w thm} to find an explicit formula for $G(y)$.
		Set 
		\begin{equation*} F(x) = \exp(-2I(x)) \; \text{and} \;G(y) = \tanh(\frac{w(0,y)}{2}).
		\end{equation*}
		By Theorem \ref{w thm}, the function $F$ satisfies
		\begin{equation}\label{FODE}
			\frac{(F'(x))^{2}}{F^2(x)} = AF^2(x) + B+ \frac{C}{F^{2}(x)},
		\end{equation}
		for some constants $A, B, C$. We next determine these constants.
		
		Observe first that $I(x)$ can be computed explicitly:
		\begin{align}
			I(x)&=\int_0^x\sin \theta(s)ds=\int_0^x \sn(cs+c_1|\frac{4}{c^2})ds \label{I(x)def} \\
			&=\frac{1}{2}\log\frac{c\;\dn(c_1|\frac{4}{c^2})+2\cn(c_1|\frac{4}{c^2})}{c\;\dn(cs+c_1|\frac{4}{c^2})+2\cn(cs+c_1|\frac{4}{c^2})}. \notag
		\end{align}
		This implies that
		\begin{align}
			F(x)&=e^{-2I(x)}=\frac{c\;\dn(cs+c_1|\frac{4}{c^2})+2\cn(cs+c_1|\frac{4}{c^2})}{a} \label{Felliptic1},\\
			\frac{1}{F(x)}&=e^{2I(x)}=-\frac{c\;\dn(cs+c_1|\frac{4}{c^2})-2\cn(cs+c_1|\frac{4}{c^2})}{b}\label{Felliptic2}.
		\end{align}
		Differentiating $F(x)=e^{-2I(x)}$ and using \eqref{sn-theta}, we obtain $F'(x)/F(x)=-2\sn(cs+c_1|\frac{4}{c^2})$. Plugging this quotient into the left hand side of \eqref{FODE} and using the elliptic expressions (\ref{Felliptic1}) and (\ref{Felliptic2}) for the right hand side, we obtain
		$A = - \frac{a^{2}}{4}$, $B = \frac{c^{2} + 4}{2}=\frac{8-ab}{2}$ and
		$C = - \frac{b^{2}}{4}$. 
		
		Finally, by Theorem \ref{w thm}, the function $G(y)$ satisfies the equation
		\begin{equation}\label{GODE}
			(G'(y))^{2} = \frac{b^{2}}{4}G^{4}(y) + \frac{ab}{2}G^{2}(y) +
			\frac{a^{2}}{4}= \frac{1}{4}(bG^{2}(y) + a )^{2}.
		\end{equation}
		Choose
		\begin{equation}\label{Gwithk}
			G(y) = -\sqrt{\frac{a}{b}}
			\tan(\frac{\sqrt{ab}}{2}y +k), 
		\end{equation}
		where $\tanh \frac{w(0,0)}{2}=- {\sqrt{\frac{a}{b}}\tan(k)}$. Then, one can verify that $w(x,y)$ satisfies the second equation (\ref{Back2Sec5}) of the B{\"a}cklund transformation, so we have found the  associated $w(x,y)$ to $\theta(x)$.
	\end{proof}
	
	Having calculated $\theta(x)$ and $w(x,y)$, we are now ready to compute the corresponding harmonic map by calculating the integrals $I_{1}$, $I_{2}$, $I_{3}$ and $I_{4}$ in Proposition \ref{Theorem2}.
	
	\begin{proof}[Proof of Theorem \ref{thetaxthm}]
		First, observe that (\ref{cndn-theta}) with (\ref{Felliptic1}) and (\ref{Felliptic2}) imply that
		\begin{align}
			F(x)&=\frac{\theta'(x)+2\cos\theta(x)}{a}, \label{Ftheta}\\
			\frac{1}{F(x)}&=-\frac{\theta'(x)-2\cos\theta(x)}{b}.\label{Finvtheta}
		\end{align}
		
		Let us start from $I_1$. Since $I'(x)=\left( \int_{0}^{x}\sin\theta(s)ds \right)'=\sin \theta(x)$ and $\tanh\frac{w(x,y)}{2}=e^{-2I(x)}G(y)$, applying an elementary trigonometric equality we obtain
		\begin{align*}
			I_{1}(x) &= \int_{0}^{x}\cosh w(t,0)\sin\theta(t)dt = 
			\int_{0}^{x} \frac{\exp(4I(t)) + G^{2}(0)}{\exp(4I(t)) - G^{2}(0)}dI(t)\\
			&=\frac{1}{2} \int_{1}^{\exp(2I(x))} \left(\frac{2u}{u^{2} - G^{2}(0)} - \frac{1}{u}\right)du, 
		\end{align*}
		or
		\begin{equation}\label{exp(2I1)}
			\exp(2I_{1}(x)) = \frac{1 - G^{2}(0)F^{2}(x)}{F(x)(1 - G^{2}(0))}=\frac{\frac{1}{F(x)} - G^{2}(0)F(x)}{1 - G^{2}(0)}.
		\end{equation}
		Using (\ref{Finvtheta}), (\ref{Ftheta}) and (\ref{Gwithk}), we conclude
		\begin{equation}\label{exp2I1theta}
			\exp(2I_{1}(x)) = 
			\frac{2\cos\theta(x)\cos(2k) - \theta'(x)}{2\cos\theta(0)\cos(2k) - \theta'(0)}.
		\end{equation}
		
		Next, we compute $I_2$. Once again, using that $\tanh\frac{w(x,y)}{2}=F(x)G(y)$ as well as an elementary trigonometric equality, we obtain
		\begin{align*}
			I_{2}(x,y) &=\cos\theta(x) \int_{0}^{y}\sinh(w(x,s))ds\\ 
			&=2\cos\theta(x)F(x) \int_{0}^{y} \frac{
				G(s)}{1 - F^{2}(x)G^{2}(s)}ds.
		\end{align*}
		An application of (\ref{GODE}) and a change of variables, yield
		\begin{align*}
			I_{2}(x,y) &= 4\cos\theta(x)F(x) \int_{0}^{y} \frac{
				G(s)}{1 - F^{2}(x)G^{2}(s)} \; \frac{G'(s)}{bG^{2}(s) + a}ds\\
			&=\frac{ 2\cos\theta(x)F(x)}{aF^{2}(x) + b}
			\log(\frac{1 - F^{2}(x)G^{2}(0)}{1 - F^{2}(x)G^{2}(y)}
			\;\frac{bG^{2}(y) + a}{bG^{2}(0) + a}
			)\\
			&=\frac{1}{2}\log(\frac
			{1 - F^{2}(x)G^{2}(y)}
			{1 - F^{2}(x)G^{2}(0)}
			\frac{bG^{2}(0) + a}
			{bG^{2}(y) + a}
			), 
		\end{align*}
		where in the last step we used \eqref{Ftheta}, \eqref{thetaODE} and the fact that $ab=4-c^2$. A manipulation similar to the one used to derive \eqref{exp2I1theta} therefore gives
		\begin{equation}\label{expI2}
			\exp(2I_{2}(x,y)) = \frac
			{2\cos\theta(x)\cos(\sqrt{ab}y + 2k) - \theta'(x)}
			{2\cos\theta(x)\cos(2k) - \theta'(x)}.
		\end{equation}
		Then, by Proposition \ref{Theorem2}, we have
		\begin{align*}
			S(x,y) &= S(0,0)\exp(2I_{1}(x))\exp(2I_{2}(x,y))\\
			&=S(0,0)\frac{2\cos\theta(x)\cos(\sqrt{ab}y + 2k) - \theta'(x)}
			{2\cos\theta(0)\cos(2k) - \theta'(x)}.
		\end{align*}

		We proceed to  $I_3$. Due to (\ref{I(x)def}), it can be written as
		\begin{align*}
			I_{3}(x) &= \int_{0}^{x} \exp(2I_{1}(t))\cosh(w(t,0))\cos\theta(t)dt\\
			&=\int_{0}^{x} \exp(2I_{1}(t))\cosh(w(t,0))\cot\theta(t)dI(t). 
		\end{align*}
		
		On the one hand, since $F(x) = e^{-2I(x)}$ satisfies (\ref{FODE}) with $A = - \frac{a^{2}}{4}$, $B = \frac{8-ab}{2}$ and
		$C = - \frac{b^{2}}{4}$, 
		we get
		$$
		\sin\theta(x) = I'(x) = -\frac{1}{2}\frac{F'(x)}{F(x)}=\frac
		{\sqrt
			{2(8-ab)e^{4I(x)} - a^{2} - b^{2}e^{8I(x)}}
		}
		{4e^{2I(x)}},
		$$
		therefore
		$$
		\cos\theta(x) = \sqrt{1 - I'(x)^{2}} = \frac{be^{4I(x)} + a}{4e^{2I(x)}}. 
		$$
		On the other hand, replacing $e^{2I_1(x)}$ by (\ref{exp(2I1)}), we can write \begin{align*}
			I_{3}(x) &= \int_{0}^{x} \frac{\exp(4I(t)) + G^{2}(0)}{(1 - G^{2}(0))\exp(2I(t))}\frac{be^{4I(x)} + a}
			{\sqrt
				{2(8-ab)e^{4I(x)} - a^{2} - b^{2}e^{8I(x)}}
			}dI(t)\\
			&=\frac{1}{2(1 - G^{2}(0))}
			\int_{1}^{\frac{2\cos\theta(x) - \theta'(x)}{b}}
			\frac{u^{2} + G^{2}(0)}{u^{2}}
			\frac{bu^{2} + a}
			{\sqrt{2(8-ab)u^{2} - a^{2} - b^{2}u^{4}}}
			du, 
		\end{align*}
		where we used the change of variables $u=e^{2I(x)}=1/F(x)$ and (\ref{Finvtheta}).
		Thus, taking into account (\ref{Gwithk}), we find
		\[ I_3(x)=\frac{1}{2}\cosh^2\frac{w(0,0)}{2}\;J\left(\frac{2\cos\theta(x) - \theta'(x)}{b}\right), 
		\]
		where $J(t)=\int_{1}^{t} \frac{u^{2} + G^{2}(0)}{u^{2}}
		\frac{bu^{2} + a}
		{\sqrt{2(8-ab)u^{2} - a^{2} - b^{2}u^{4}}}
		du.$ Note that $J$ can be computed by the use of elliptic integrals.
		
		Finally, for the last integral $I_4$ we have
		\begin{align*}
			I_{4}(x,y) &= e^{2I_{1}(x)}\sin\theta(x)\int_{0}^{y}\exp(2I_{2}(x,s))\sinh(w(x,s))ds\\
			&=\frac{1}{2}e^{2I_{1}(x)}\tan\theta(x)\int_{0}^{y}
			e^{(2\cos\theta(x)\int_{0}^{s}\sinh(w(x,\psi))d\psi)}2\cos\theta(x)\sinh(w(x,s))ds\\
			&= \frac{1}{2}e^{2I_{1}(x)}\tan\theta(x)(e^{2I_{2}(x,y)} - e^{2I_{2}(x,0)}) \\
			&=  \frac{1}{2}e^{2I_{1}(x)}\tan\theta(x)(e^{2I_{2}(x,y)} - 1)\\
			&=\frac{\sin\theta(x)(\cos(\sqrt{ab}y + 2k) - \cos(2k))}
			{2\cos\theta(0)\cos(2k) - \theta'(0)}, 
		\end{align*}
		where in the last step we used (\ref{exp2I1theta}) and (\ref{expI2}). The calculation of $R(x,y) = R(0,0) + 2S(0,0)(I_{3}(x) - I_{4}(x,y))$ is now complete.
	\end{proof}
	Next, using Theorem \ref{thetaxthm}, we provide an entirely new harmonic map to the hyperbolic plane.
	
	\begin{example}
		Consider the case when $\theta(0) = 0$, $\theta'(0) = 1$, $G(0) = 0=\tanh\frac{w(0,0)}{2}$, $S(0,0) = 1$ and $R(0,0) = 0$. Then, we compute
		\[\sin\theta(x) = \sn(x|4) = \frac{1}{2} \sn (2x|\frac{1}{4}),\]
		\[\tanh(\frac{w(x,y)}{2}) = \frac{\sqrt{3}\tan(\frac{\sqrt{3}}{2}y)}{\theta'(x) - 2\cos\theta(x)},\]
		and
		$$
		I(x) = \frac{1}{2}\log(2\cos\theta(x) - \theta'(x)).
		$$
		Therefore, 
		\begin{align*}
			I_{1}(x) &= I(x),\\
			I_{2}(x, y) &= \frac{1}{2}\log(\frac{2\cos\theta(x)\cos(\sqrt{3}y) - \theta'(x)}
			{2\cos\theta(x) - \theta'(x)}),\\
			I_{3}(x) &= \frac{1}{2}\int_{1}^{2\cos\theta(x) - \theta'(x)} \frac{u^{2} + 3}{\sqrt{10u^{2} - u^{4} - 9}}du,\\
			I_{4}(x,y) &= \sin\theta(x)(\cos(\sqrt{3}y) - 1).
		\end{align*}
		Then, the corresponding harmonic map then is $u(x,y) = R(x,y) + iS(x,y)$, where 
		$$
		R(x,y) = J(2\cos\theta(x)-\theta'(x))
		- 2\sin\theta(x)(\cos(\sqrt{3}y) - 1)
		$$
		and
		$$
		S(x,y) = 2\cos\theta(x)\cos(\sqrt{3}y) - \theta'(x),
		$$
		where $J(t)=\int_{1}^{t} \frac{u^{2} + G^{2}(0)}{u^{2}}
			\frac{bu^{2} + a}
			{\sqrt{2(8-ab)u^{2} - a^{2} - b^{2}u^{4}}}
			du.$
		
		The domain of definition of $u$ is the set of \[\left\{ (x,y):  \; \left|\frac{\sqrt{3}\tan(\frac{\sqrt{3}}{2}y)}{\theta'(x) - 2\cos\theta(x)}\right|<1\right\}\] where the map is a well defined $C^{2}$ map whose Jacobian is almost everywhere non vanishing. 
	\end{example}
	
	\begin{acknowledgement}
			The authors would like to sincerely thank the referees for their valuable comments, which helped improve the initial manuscript.
		\end{acknowledgement}
	
	\section*{Funding}
	E. Papageorgiou  is supported by the Hellenic Foundation for Research and Innovation, Project HFRI-FM17-1733.
	
	\section*{Conflicts of interest/Competing interests}
	Financial interests: The authors declare they have no financial interests. Non-financial interests: none.

	%%%%%%%%%%% Affiliation
	\vspace{20pt}
	\address{
		\noindent\textsc{Giannis Polychrou:}
		\href{mailto:ipolychr@math.auth.gr}
		{ipolychr@math.auth.gr}\\
		Department of Mathematics, Aristotle University of Thessaloniki,
		Thessaloniki 54124, Greece}
	
	\vspace{10pt}
	\address{
		\noindent\textsc{Effie Papageorgiou:}
		\href{mailto:papageoeffie@gmail.com}
		{papageoeffie@gmail.com}\\
		Department of Mathematics and Applied Mathematics,
		University of Crete,
		Crete 70013, Greece}
	
	\vspace{10pt}
	\address{
		\noindent\textsc{Anestis Fotiadis}
		\href{mailto:fotiadisanestis@math.auth.gr}
		{fotiadisanestis@math.auth.gr}\\
		Department of Mathematics, Aristotle University of Thessaloniki,
		Thessaloniki 54124, Greece}
	
	\vspace{10pt}
	\address{
		\noindent\textsc{Costas Daskaloyannis}
		\href{mailto:daskalo@math.auth.gr}
		{daskalo@math.auth.gr}\\
		Department of Mathematics, Aristotle University of Thessaloniki,
		Thessaloniki 54124, Greece}
\end{document}